\documentclass[reqno]{amsart}
\usepackage{enumerate,color,amssymb}
\usepackage{url}

\newtheorem{theorem}{Theorem}[section]
\newtheorem{lemma}[theorem]{Lemma}
\newtheorem{proposition}[theorem]{Proposition}

\renewcommand{\Re}{\operatorname{Re}}
\renewcommand{\Im}{\operatorname{Im}}
\newcommand{\li}{\operatorname{li}}
\newcommand{\Sclass}{{\mathcal{S}}}

\newcommand{\ve}{\varepsilon}

\begin{document}
\title{Large values of L-functions from the Selberg class}

\author{Christoph Aistleitner}
\address{Institute of Analysis and Number Theory, TU Graz, Steyrergasse 30/II, 8010 Graz, Austria}
\email{aistleitner@math.tugraz.at}

\author{{\L}ukasz Pa\'nkowski}
\address{Faculty of Mathematics and Computer Science, Adam Mickiewicz University, Umultowska 87, 61-614 Pozna\'{n}, Poland, and Graduate School of Mathematics, Nagoya University, Nagoya, 464-8602, Japan}
\email{lpan@amu.edu.pl}

\thanks{The first author was supported by the Austrian Science Fund (FWF), projects I1751-N26, F5507-N26 and Y-901-N35. The second author was partially supported by (JSPS) KAKENHI grant no. 26004317 and the grant no. 2013/11/B/ST1/02799 from the National Science Centre.}



\begin{abstract}
In the present paper we prove lower bounds for L-functions from the Selberg class, by this means improving earlier results obtained by the second author together with J\"orn Steuding. We formulate two theorems which use slightly different technical assumptions, and give two totally different proofs. The first proof uses the ``resonance method'', which was introduced by Soundararajan, while the second proof uses methods from Diophantine approximation which resemble those used by Montgomery. Interestingly, both methods lead to roughly the same lower bounds, which fall short of those known for the Riemann zeta function and seem to be difficult to be improved. Additionally to these results, we also prove upper bounds for L-functions in the Selberg class and present a further application of a theorem of Chen which is used in the Diophantine approximation method mentioned above.
\end{abstract}

\maketitle

\section{Introduction.}

It is well known that the absolute value of the Riemann zeta function $\zeta(\sigma+it)$ takes arbitrarily large and arbitrarily small values when $t$ runs through the real numbers and $\sigma\in [1/2,1)$ is a fixed real number. However, the growth of the Riemann zeta function as a function of $t$ (for fixed $\sigma$) cannot be too fast, since its absolute value is bounded by a power of $t$. More precisely, if $\mu_{\zeta}(\sigma)$ denotes the infimum over all $c\geq 0$ satisfying $\zeta(\sigma+it)\ll t^{c}$ for sufficiently large $t$, then one can show that $\mu_{\zeta}(\sigma)\leq (1-\sigma)/2$ for $0\leq \sigma\leq 1$. Although the upper bound for $\mu_{\zeta}(\sigma)$ has been improved by many mathematicians, especially for $\sigma = 1/2$, it is yet unproved (but widely believed) that $\mu_{\zeta}(\sigma)=0$ (for more details we refer to \cite{I} or \cite{T}). As evidence for the truth of this conjecture one can regard the Riemann hypothesis, which implies that
\[
\log\zeta(\sigma+it)\ll \frac{(\log t)^{2-2\sigma}}{\log\log t},\qquad \text{for} \quad  \frac{1}{2} \leq \sigma <1.
\]

Therefore, it is natural to ask for omega results on $\zeta(\sigma+it)$. The first answer was given by Titchmarsh (see \cite[Theorem 8.12]{T}), who proved that for any $\sigma\in [1/2,1)$ and every $\varepsilon>0$ the inequality $|\zeta(\sigma+it)|>\exp\left((\log t)^{1-\sigma-\varepsilon}\right)$ holds for arbitrarily large values of $t$. In 1977, Montgomery \cite{M} improved this result for $\sigma\in(1/2,1)$ by proving that for any fixed $\sigma\in (1/2,1)$ and every sufficiently large $T$ there exists $t$ such that $T^{(\sigma-1/2)/3}\leq t\leq T$ and 
\begin{equation} \label{mont}
|\zeta(\sigma+it)| \geq \exp\left(\frac{1}{20}\left(\sigma-\frac{1}{2}\right)^{1/2} \frac{(\log T)^{1-\sigma}}{(\log\log T)^{\sigma}}\right).
\end{equation}
Moreover, he showed that under the Riemann Hypothesis the above inequality can be extended to $\sigma\in[1/2,1)$ with a slightly better constant and better range of $t$.

The first unconditional proof of Montgomery's theorem for $\sigma=1/2$ was given by Balasubramanian and Ramachandra \cite{BR}. The best result currently known is due to Bondarenko and Seip \cite{BS}, who very recently achieved a breakthrough by proving that 
\[
\max_{T^{1/2}\leq t\leq T}\left|\zeta\left(\frac{1}{2}+it\right)\right|\geq \exp\left(\left(\frac{1}{\sqrt{2}}+o(1)\right)\sqrt{\frac{\log T\log\log\log T}{\log\log T}}\right).
\]
Their proof is based on the so-called resonance method, which was introduced by Soundararajan \cite{So}, and on a connection between extreme values of the Riemann zeta function and certain sums involving greatest common divisors (GCD sums). This connection was discovered by Hilberdink \cite{H}. Recently, the first author \cite{A} succeeded in applying the resonance method with an extremely long resonator such that he could recapture Montgomery's results by the resonance method, off the critical line, an idea which also plays a crucial role in the omega result of Bondarenko and Seip.\\

Similar problems of finding extreme values were also investigated for other zeta and $L$-functions, and it was shown that Montgomery's approach can be applied to some generalizations of the Riemann zeta function. For example, Balakrishnan \cite{B} showed that Dedekind zeta functions take large values of order $\exp(c (\log T)^{1-\sigma}/(\log\log T)^\sigma)$, and Sankaranarayanan and Sengupta \cite{SS} generalized Montgomery's theorem to a wide class of $L$-functions defined by Dirichlet series with real coefficients under some natural analytic and arithmetic conditions.

Recently, the second author and Steuding \cite{PS} investigated further refinements of Montgomery's reasoning and proved that for every $L$-function $L(s)=\sum_{n\geq 1}a_L(n)n^{-s}$ from the Selberg class which satisfies $L(s)\ne 0$ for $\sigma>1/2$ we have
\begin{equation} \label{pan-st}
\max_{t\in [T,2T]} |L(\sigma+it)|\geq \exp\left(c\frac{(\log T)^{1-\sigma}}{(\log\log T)^{2-\sigma}}\right)
\end{equation}
for some explicitly given constant $c>0$ and sufficiently large $T$, under the additional assumption that the coefficients of $L$ satisfy a prime number theorem with remainder term in the form
\begin{equation}\label{SPNT}
\sum_{p\leq x}|a_L(p)| = \kappa\frac{x}{\log x} + \mathcal{O}\left(\frac{x}{\log^2 x}\right),\qquad (\kappa >0).
\end{equation}
Note that Montgomery's argument requires a prime number theorem in order to get a lower bound for the sum of $|a_L(p)|$ over primes in some interval, which might be estimated from below by the sum of $|a_L(p)|^2$, provided $|a_L(p)|\ll 1$. Hence, the condition \eqref{SPNT} can be replaced by the more natural assumption that $L$ has a polynomial Euler product and satisfies the Selberg normality conjecture in the stronger form
\begin{equation}\label{SNC}
\sum_{p\leq x}|a_L(p)|^2 = \kappa\frac{x}{\log x} + \mathcal{O}\left(\frac{x}{\log^2 x}\right),\qquad (\kappa >0).
\end{equation}

The main difference between Montgomery's proof and the proof in \cite{PS} is that due to the appearance of non-real coefficients $a_L(n)$ one requires an \emph{inhomogeneous} Diophantine approximation theorem (which is used in an effective form due to Weber \cite{W}), while in the case of the Riemann zeta function one has positive real coefficients and can use classical results from \emph{homogeneous} Diophantine approximation. This difference also explains the fact why \eqref{pan-st} has a worse exponent of the $\log \log$ term inside the exponential function than the one appearing in \eqref{mont}.\\

Recall that the Selberg class $\Sclass$ consists of those functions $L(s)$ defined by a Dirichlet series $\sum_{n=1}^\infty a_L(n)n^{-s}$ in the half-plane $\Re s > 1$ which satisfy the following axioms:
\begin{enumerate}[(i)]
\item [(i)] {\it Ramanujan hypothesis:} $a_L(n)\ll_\varepsilon n^\varepsilon$ for every $\varepsilon>0$;
\item [(ii)] {\it analytic continuation:} there exists a non-negative integer $m$ such that $(s-1)^mL(s)$ is an entire function of finite order;
\item [(iii)] {\it functional equation:} $L(s)$ satisfies the following functional equation
\[
\Lambda(s) = \theta\overline{\Lambda(1-\overline{s})},
\]
where
\[
\Lambda(s):=L(s)Q^s\prod_{j=1}^k\Gamma(\lambda_js+\mu_j),
\]
$|\theta|=1$, $Q\in\mathbb{R}$, $\lambda_j\in\mathbb{R}_+$, and $\mu_j\in\mathbb{C}$ with $\Re \mu_j \geq 0$;
\item [(iv)] {\it Euler product:} for prime powers $p^j$ there exist complex numbers $b(p^j)$ such that for $\Re s>1$ we have 
$$
L(s) = \prod_p L_p(s), \qquad \textrm{where} \qquad L_p(s)=\exp \left( \sum_{j=1}^\infty \frac{b(p^j)}{p^{js}} \right),
$$
and such that $b(p^j) \ll p^{j \delta}$ for some $\delta < 1/2$. 
\end{enumerate}
~\\

Some of our results will be for the \emph{Selberg class with polynomial Euler product} denoted by $\mathcal{S}'$, for which axiom (iv) in the list above is replaced by axiom (iv') below: 
\begin{enumerate}
\item [(iv')] {\it polynomial Euler product:} for $\Re s>1$ we have
\begin{equation}\label{eq:EulerProd}
L(s) = \prod_p\prod_{j=1}^m\left(1-\frac{\alpha_j(p)}{p^s}\right)^{-1},
\end{equation}
where $\alpha_j(p)$ are complex numbers.
\end{enumerate}
~\\

It is easy to notice that under axiom (iv') for any positive integers $\beta_j$ we have
\[
a_L\left(\prod_{j=1}^n p_j^{\beta_j}\right) =\prod_{j=1}^n\sum_{\substack{k_1,\ldots,k_m\geq 0\\k_1+\ldots+k_m=\beta_j}}\prod_{i=1}^m\alpha_i(p_j)^{k_j}.
\]
Hence (see \cite[Lemma 2.2]{S}) under axiom (iv') the axiom (i) is equivalent to the assumption that $|\alpha_j(p)|\leq 1$ for $j=1,2,\ldots,m$ and every prime $p$. Moreover, one can easily observe that in this case
\[
|a_L(p)|<m\qquad\text{for all primes $p$},
\]
where $m$ is closely related to the degree $d_L = 2\sum_{j=1}^m \lambda_j$, since it is widely believed that most important $L$-functions satisfy a functional equation with $k=m$ and $\lambda_j = 1/2$. We will need to work with $\Sclass'$ instead of the original Selberg class $\Sclass$ only in the application of the resonance method, which requires bounded coefficients $a_L(p)$. However, it seems that this restriction does not exclude any important $L$-function in number theory, since so far all known examples of $L$-functions like the Riemann zeta function, Dirichlet $L$-functions, Hecke
L-functions, as well as the (normalized) L-functions of
the holomorphic modular form and general automorphic L-functions
have polynomial Euler products and, at least under some widely believed conjectures, are elements of $\Sclass'$. Moreover, it should be noted that although other results in the paper are stated under the assumption $L\in\Sclass$ and some version of \eqref{SPNT} holds, one can easily show that they can be proved in a slightly weaker form if we replaced these assumptions by assuming that $L\in\Sclass'$ and 
some variant of \eqref{SNC} is true, which seems to be a slightly more natural requirement than~\eqref{SPNT}. \\

It is well known that using the Phragm{\'e}n-–Lindel\"of principle one can show (see, for example, \cite[Theorem 6.8]{S}) that, similarly to the case of the Riemann zeta function, all $L$-functions from the Selberg class have polynomial order of growth inside the critical strip, namely
\begin{equation}\label{fromPhL}
L(\sigma+it) \ll t^{\frac{d_L}{2}(1-\sigma)+\varepsilon},\qquad 0\leq \sigma\leq 1,\ t\geq t_0>0,
\end{equation}
where $d_L$ denotes the degree of $L$ and the numbers $\lambda_j$ are defined by the functional equation satisfied by $L(s)$. Moreover, it is conjectured that all elements of the Selberg class satisfy an analogue of the Riemann Hypothesis, which can be used to give the upper bound on the maximal values taken by $L(s)$ stated below.

\begin{proposition}\label{prop:1}
Assume that $\Sclass \ni L(s)\ne 0$ in the half-plane $\Re s>1/2$ and satisfies 
\begin{equation}\label{prime_prop}
\sum_{p \leq x} |a_L(p)| = (\kappa+o(1)) \frac{x}{\log x} \qquad(\kappa>0).
\end{equation}
Then for every fixed $\sigma \in [1/2,1)$ there exists a constant $c>0$ such that
\[
L(\sigma+it) \ll \exp\left(c\frac{(\log t)^{2-2\sigma}}{\log\log t}\right).
\]
Moreover, if we assume \eqref{SPNT}, we have
\[
L(1+it)\ll (\log\log t)^\kappa.
\]
\end{proposition}

The main results of the present paper are the following two theorems on lower bound for large values of $L$-functions. We will first state both of them, and then comment on their relation to each other and on their proofs.

\begin{theorem}\label{th:1}
Assume that $L=\sum_{n\geq 1}a_L(n)n^{-s} \in\Sclass'$ satisfies the following Selberg's normality conjecture,
\begin{equation}\label{eq:SNCthm}
\sum_{p\leq x}|a_L(p)|^2 = (\kappa+o(1))\frac{x}{\log x}\qquad(\kappa>0).
\end{equation}
Then for every fixed $\sigma\in [1/2,1)$ and sufficiently large $T$ we have
\[
\max_{t\in[T,2T]} |L(\sigma+it)|\geq \exp\left( \left(C_L(\sigma)+o(1)\right)\frac{(\log T)^{1-\sigma}}{(\log\log T)^{\theta(\sigma)}}\right),
\]
where $\theta(1/2)=1/2$ and $\theta(\sigma)=1$ otherwise, and where
\[
C_L(\sigma) = 
\begin{cases}
\kappa^\sigma m^{1-2\sigma}\frac{(3-2\sigma)^{3/2-\sigma}}{2 (2\sigma-1)^{1/2}}&\text{ if $\frac{1}{2}<\sigma<1$};\\
\sqrt{\kappa}&\text{ if $\sigma=\frac{1}{2}$}.
\end{cases}
\]
\end{theorem}

In the statement of the following theorem, we write $N_0(\sigma,T)$ for the number of non-trivial zeros $\rho = \beta + i \gamma$ which have real part $\beta > \sigma$ and imaginary part $\gamma \in
(0,T]$.

\begin{theorem}\label{th:2}
Let $L(s)$ be an element of the Selberg class, which for $\Re s > 1$ is defined
by the Dirichlet series $\sum_{n \geq 1} a_L(n) n^{-s}$. Let real numbers
$\theta$ and $\sigma \in [1/2,1)$ be given. Assume that there exists a number
$\eta>0$ such that
\begin{equation} \label{zerod}
N_0 (\sigma,T)  \ll T^{1-\eta},
\end{equation}
and that
\begin{equation} \label{prime}
\sum_{p \leq x} |a_L(p)| = (\kappa+o(1)) \frac{x}{\log x} \qquad(\kappa>0).
\end{equation}
Then for every sufficiently large $T$ we have
$$
\max_{t \in [T,2T]} \Re e^{-i \theta} \log L(\sigma+it) \geq c_{\kappa,\eta}
\frac{(\log T)^{1-\sigma}}{\log \log T},
$$
where
$$
c_{\kappa,\eta} = \frac{(1-e^{-1})\kappa}{4} \left(\frac{\eta}{4 \sqrt{e}} \right)^{1-\sigma}.
$$ 
\end{theorem}

Theorem \ref{th:1} is proved using the resonance method in the spirit of \cite{So}, while Theorem \ref{th:2} is proved using a variant of Montgomery's method using Diophantine approximation as in \cite{PS}. The necessary assumptions reflect the strong and weak points of each of these respective methods. For the application of the resonance method we need a strong upper bound on the size of the coefficients, but no information on the zeros of the $L$-function. Furthermore, the resonance method gives a better result for $\sigma=1/2$, which is quite natural given the observations in \cite{H,So}. On the other hand, for the application of the method using Dirichlet approximation no bound on the size of the coefficients is necessary, but (since one changes from $L$ to $\log L$) some knowledge on the zeros of the $L$-function is required. It should be noted that proving the kind of zero-density estimate which is assumed in Theorem \ref{th:2} is usually extremely difficult for a generic $L$-function unless $\
\sigma$ is very close to 1. This problem is closely related to the problem of giving upper bounds for the order of magnitude of $L$-functions, which as we have seen in \eqref{fromPhL} depends on $d_L$ and is known to be small (even in the mean-square sense) only when $\sigma$ is close to $1$.\\

It is quite remarkable that both methods mentioned above come to their limit at $\exp\left(c \frac{(\log T)^{1-\sigma}}{(\log\log T)}\right)$ for $\sigma \in (1/2,1)$, and that in both cases it seems to be a very difficult problem to get beyond this barrier. We conjecture that for $\sigma \in (1/2,1)$ under assumptions 
such as those in the statements of Theorem \ref{th:1} and \ref{th:2} one should actually be able to achieve roughly the same lower bounds as in the case of the Riemann zeta function \eqref{mont}, that is, with the power of the $\log \log$ term in the denominator of the exponential function reduced from 1 to $\sigma$. However, as noted, this seems to be very difficult. In both proofs it is clearly visible why it is not possible to go significantly beyond the result obtained in Theorems \ref{th:1} and \ref{th:2}. In the case of the resonance method the restriction essentially comes from the requirement of keeping the length of the resonator well below $T$, and the problem corresponds to successfully implementing the ``extremely long resonator'' in this setting, which is prohibited by the fact that the coefficients are not necessarily positive real numbers and thus a certain ``positivity'' property, which plays a crucial role in \cite{A,BS}, fails to hold. In the case of the method using Diophantine approximation, the problem corresponds to 
bounding away linear forms of logarithms of primes from the origin; instead of the general Diophantine approximation results used in the current proof of Theorem \ref{th:2} one would need to use a Diophantine approximation tool which is tailor-made for dealing with logarithms of primes, and take into account the Diophantine properties of these logarithms of primes.\\

The precise relation between the resonance method and the Diophantine approximation method seems to be not really understood yet, which is the reason for including both proofs in the present paper, even if they come to rather similar conclusions. It is interesting to compare the restrictions which prevent further improvements of the proofs of Theorems \ref{th:1} and \ref{th:2} given below. In the case of the proof of Theorem \ref{th:1}, the restriction comes from the requirement of keeping the length of the resonator well below $T$, to make sure that the contribution of the off-diagonal terms is negligible. These off-diagonal terms contain quotients $(m/n)^{it}$, where $m$ and $n$ have non-zero coefficients in the representation of the resonator as a Dirichlet polynomial, and it is important that these quotients are bounded away from 1 to make sure that $\int_{T}^{2T} (m/n)^{it}~dt$ is small. In the case of Theorem \ref{th:2}, the restriction comes from the quantity $\Lambda$ which appears in Lemma \ref{lemmachen}
 below, and which is defined as the minimal value of a linear form in logarithms of primes. Thus in a sense the restriction is of a similar nature in both proofs, just that it appears once in multiplicative form (as the quotient of terms in the resonator which has to be bounded away from 1) and once, after taking logarithms, in additive form (as a linear form of logarithms of primes which has to be bounded away from 0). What is furthermore striking is the fact that the lemma from inhomogeneous Diophantine approximation, which plays the crucial role in our proof of Theorem \ref{th:2}, is itself proved (see \cite{chen}) using a probabilistic method which has some resemblance of a ``resonance'' argument. A further remark concerns the paper of Balasubramanian and Ramachandra \cite{BR}, whose method at first glance looks very different from the other two methods mentioned here so far. However, on second thought one is tempted to read their method as a kind of resonance argument where a high moment of $\zeta$ 
plays the role of the resonator. This would fit together with the way how the resonators are constructed in \cite{So} and in subsequent papers, namely as multiplicative functions which are supported on numbers having many small prime factors. The bottom line is that the connections between all 
these 
methods are not well understood, and that there would be some merit in clarifying these connections.\\

The results obtained in Theorems \ref{th:1} and \ref{th:2} should be compared to the earlier result in \eqref{pan-st}. Note that in both cases the exponent of the $\log \log$ term inside the exponential function is reduced from $2-\sigma$ to 1 for $\sigma \in (1/2,1)$, and that in Theorem …\ref{th:1} it is reduced from 1 to 1/2 for $\sigma=1/2$. Note again that for Theorem \ref{th:1} no analogue of the Riemann Hypothesis has to be assumed; for Theorem \ref{th:2} the assumption of an analogue of the Riemann hypothesis in \cite{PS} is reduced to the assumption of a zero density estimate. Furthermore, for both theorems the conditions~\eqref{SPNT} and~\eqref{SNC} have been relaxed to the assumption of the Selberg normality conjecture and a prime number theorem for the coefficients $a_L(n)$, respectively, without an upper bound on the order of the error term. The proof of Theorem \ref{th:1} is inspired by the proofs given in \cite{So}. The proof of Theorem \ref{th:2} follows the one in \cite{PS}, but instead of 
using an $\ell^\infty$ (maximal error) result from inhomogeneous Diophantine approximation it uses an $\ell^2$ (average error) result of Chen \cite{chen}. 

Chen's theorem can also be used to improve the lower bound for extreme values on the line $\sigma=1$. The Riemann zeta-function case is well-investigated and the best known result is due to Granville and Soundararajan \cite{GS}, who showed that
\[
\max_{T\leq t\leq 2T}|\zeta(1+it)| \geq e^\gamma(\log\log T+\log\log\log T - \log\log\log\log T + \mathcal{O}(1)),
\]
which is the expected order of magnitude under the Riemann Hypothesis; here $\gamma$ denotes the Euler constant. As we showed in Proposition \ref{prop:1} above a similar order of magnitude is expected for $L$-functions, whereas the best known result concerning Dirichlet $L$-functions $L(s,\chi)$ is due to Steuding \cite{S2}. Using an effective version of Kronecker's theorem, he proved the existence of infinitely many $s=\sigma+it$ with $\sigma\to 1^+$ and $t\to+\infty$ such that $|L(\sigma+it,\chi)|$ is of size at least $\log\log\log t/\log\log\log\log t$, which, by the Phragm\'en--Lindel\"of principle, leads to 
\[
|L(1+it,\chi)| = \Omega\left(\frac{\log\log\log t}{\log\log\log\log t}\right).
\]

An application of Chen's lemma allows to prove the following refinement of Steuding's result, which for $\sigma=1$ gives the expected order of magnitude.

\begin{theorem} \label{th:3}
Let $L(s)\in \Sclass$ be defined by the Dirichlet series $\sum_{n\geq 1}a_L(n) n^{-s}$ for $\Re s>1$. Assume that \eqref{SPNT} holds and that $\theta$ is an arbitrary given real number. Then for every sufficiently large $T$, there is $\sigma>1$ and $t\in [T,2T]$ such that
\begin{equation}\label{eq:Steuding}
\Re e^{-i\theta} \log L(\sigma+it) \geq \kappa\log \log \log T +\mathcal{O}(1).
\end{equation}
In particular, 
\[
|L(1+it)| = \Omega((\log\log t)^\kappa).
\]
\end{theorem}

The remaining part of this paper is organized as follows. In Section \ref{sec_th1} we give the proof of Theorem \ref{th:1} using the resonance method. In Section \ref{sec_th2} we prove Theorem \ref{th:2} using methods from inhomogeneous Diophantine approximation. In Section \ref{sec_prop1} we comment on possible further improvements of the method from Section \ref{sec_th1}, and we prove the upper bound given in Proposition \ref{prop:1}. Finally, in Section \ref{sec_th3th4}, we give the proof of Theorems \ref{th:3}.\\

\section{Approximation of $L$-function by a Dirichlet polynomial and the resonance method.} \label{sec_th1}

Throughout this section, we assume that the assumptions of Theorem \ref{th:1} are satisfied. In particular we assume that $L(s)=\sum_{n\leq 1}a_L(n)n^{-s}$ denotes a fixed element of the Selberg class with polynomial Euler product, and that $\sigma$ is a fixed real number from the interval $[1/2,1)$.\\

First we will show that a given $L$-function can be approximated by a corresponding Dirichlet polynomial with an extra smoothing factor $e^{-(n/X)}$, which will be negligible for small $n$. In order to do that, let us put $X=T^{d_L+\varepsilon}$ for some small positive $\varepsilon$. It is well known that the Mellin inversion formula for the gamma function gives
\[
e^{-(n/X)} = \frac{1}{2\pi i}\int_{2-i\infty}^{2+i\infty}\Gamma(w)n^{-w}X^w dw
\qquad (X>0).
\]
Therefore, 
\begin{equation}\label{eq:fromMellin}
\sum_{n=1}^\infty \frac{a_L(n)}{n^s}e^{-(n/X)} = \frac{1}{2\pi i}\int_{2-i\infty}^{2+i\infty}L(s+w)\Gamma(w)X^w dw.
\end{equation}
Note that, by Stirling's formula, the integration over $|\Im(w)|>\frac{T}{2} $ is bounded if $t\in [T,2T]$. Thus it suffices to consider
\begin{equation}\label{eq:contour1}
\frac{1}{2\pi i}\int_{2-\frac{i}{2}T}^{2+\frac{i}{2}T}L(s+w)\Gamma(w)X^w dw.
\end{equation}

Now we move the contour of integration in \eqref{eq:contour1} to the line $\Re w = -\sigma$. Since $\Im w+t>T/2$, we pass only the simple pole of $\Gamma(w)$ at $w=0$ where the residue of the integrand is $L(s)$. Moreover, \eqref{fromPhL} implies that $L(i(t+w))\ll T^{(d_L+\varepsilon)/2}$. Therefore, by the choice of $X$, we have
\[
\frac{1}{2\pi i}\int_{-\sigma-i\frac{T}{2}}^{-\sigma+i\frac{T}{2}}L(s+w)\Gamma(w)X^w dw\ll T^{(d_L+\varepsilon)(1/2-\sigma)}\ll 1,
\]
and, in consequence, for $t\in [T,2T]$, 
\[
L(\sigma+it)=\sum_{n=1}^\infty \frac{a_L(n)}{n^s}e^{-(n/X)} + \mathcal{O}(1).
\]

Let us observe that for $n>X(\log X)$ we have $\exp(-n/X)\leq n^{-2/3}$, so the series $\sum_{n\geq X (\log X)}\frac{a_L(n)}{n^s}e^{-(n/X)}$ is bounded and we obtain
\begin{equation}\label{eq:approxByPoly}
L(\sigma+it)=\sum_{n\leq T^{d_L+2\varepsilon}} \frac{a_L(n)}{n^s}e^{-(n/X)} + \mathcal{O}(1).
\end{equation}

Now we use \eqref{eq:approxByPoly} for the resonance method as introduced by Soundararajan \cite{So}. Following his notation, let $\Phi$ be a smooth function compactly supported on the interval $[1,2]$ and satisfying $0\leq \Phi(t)\leq 1$ and $\Phi(t)=1$ for $t\in (5/4,7/4)$. Then the Fourier transform of $\Phi$ satisfies $\hat{\Phi}(y) = \int_{\mathbb{R}}\Phi(t)e^{-ity}dt\ll |y|^{-v}$ for any fixed positive integer $v$. One can show (see (2) in \cite{So}) that for any Dirichlet polynomial $R(t) = \sum_{n\leq N} r(n)n^{-it}$ we have
\[
M_1=M_1(R,T) := \int_{\mathbb{R}}|R(t)|^2\Phi(t/T)dt =  T\hat{\Phi}(0)\left(1+\mathcal{O}(1/T)\right)\sum_{n\leq N}|r(n)|^2,
\]
provided that $N\leq T^{1-\varepsilon}$.\\

Now let us consider 
\begin{align*}
M_2=M_2(R,T,\sigma)&:=\int_{\mathbb{R}}L(\sigma+it)|R(t)|^2\Phi(t/T)dt.
\end{align*}
Obviously,
\begin{equation}\label{eq:maxLfirst}
\max_{t\in[T,2T]}|L(\sigma+it)|\gg \frac{|M_2(R,T,\sigma)|} {M_1(R,T)}.
\end{equation}

In order to estimate $M_2$ we notice that \eqref{eq:approxByPoly} implies
\[
M_2 = T\sum_{n,m\leq N}\sum_{k\leq T^{d_L+2\varepsilon}}\frac{a_L(k)r(m)\overline{r(n)}}{k^\sigma}e^{-(k/X)}\hat{\Phi}\left(T\log(mk/n)\right)+\mathcal{O}(M_1(R,T)).
\]
and that the contribution from off-diagonal terms $mk\ne n$ is 
\[
\ll\frac{N}{T}\sum_{n\leq N}|r(n)|^2\sum_{k\leq T^{d_L+\varepsilon}}\frac{|a_L(k)|}{k^\sigma}\ll T\sum_{n\leq N}|r(n)|^2 = \mathcal{O}(M_1(R,T)),
\]
since $\hat{\Phi}(T\log(mk/n))\ll T^{-d_L/2-1}$ for $mk\ne n$, $N\leq T^{1-\varepsilon}$, and since $a_L(n)\ll n^{\varepsilon'}$ for any (arbitrarily small, fixed) $\varepsilon' >0$. Hence
\begin{equation}\label{eq:I_1first}
M_2 = T\hat{\Phi}(0)\sum_{mk=n\leq N}\frac{a_L(k)r(m)\overline{r(n)}}{k^\sigma}e^{-(k/X)}+\mathcal{O}(M_1(R,T)),
\end{equation}
and it remains to prove that there is a resonator $R(t)$ such that
\[
\left|\sum_{mk=n\leq N}\frac{a_L(k)r(m)\overline{r(n)}}{k^\sigma}e^{-(k/X)}\right|\Bigg/\sum_{n\leq N}|r(n)|^2\geq
\exp\left(C_L(\sigma)\frac{(\log N)^{1-\sigma}}{\log\log N}\right).
\]

Let us put $r(n)=a_L(n)f(n)$, where $f$ is a non-negative real multiplicative function supported only on the square-free numbers. Then, for $mk=n\leq N$, 
\[
\left|\sum_{mk=n\leq N}\frac{a_L(k)r(m)\overline{r(n)}}{k^\sigma}e^{-(k/X)}\right| \geq \frac{1}{2}\sum_{mk=n\leq N}\frac{a_L(k)r(m)\overline{r(n)}}{k^\sigma},
\]
since $a_L(k)r(m)\overline{r(n)}=|a_L(k)|^2 |a_L(m)|^2 f(m)^2f(k)\in\mathbb{R}_{\geq 0}$, and $\exp(-k/X)\geq 1/2$ for $k<T^{1-\varepsilon}$ and sufficiently large $T$. Thus Theorem \ref{th:1} follows easily from the following lemma, whose proof follows the proof of \cite[Theorem 2.1]{So} with the application of \eqref{eq:SNCthm} instead of the classical prime number theorem. 

\begin{lemma}
For every $\sigma\in [1/2,1)$ and every sufficiently large $N$ there is a real multiplicative function $f(n)$ supported on the square-free numbers such that 
\begin{multline}
\sum_{mk\leq N}\frac{|a_L(k)|^2|a_L(m)|^2f(k)f(m)^2}{k^\sigma}\Bigg/\sum_{n\leq N}|a_L(n)|^2f(n)^2\\
\geq
\exp\left(C_L(\sigma)\frac{(\log N)^{1-\sigma}}{(\log\log N)^{\theta(\sigma)}}\right), \nonumber
\end{multline}
where $\theta(\sigma)$ and $C_L(\sigma)$ are the same as in Theorem \ref{th:1}.
\end{lemma}

\begin{proof}
Since $f(n)$ is supported only on the square-free numbers and since both functions $a_L(n)$ and $f(n)$ are multiplicative, we have
\begin{align}
&\left|\sum_{mk=n\leq N}\frac{a_L(k)r(m)\overline{r(n)}}{k^\sigma} \right|\nonumber\\
&\qquad\qquad\qquad\qquad=\sum_{k\leq N}\frac{f(k)|a_L(k)|^2}{k^\sigma} \sum_{\substack{m\leq N/k\\ \gcd(k,m)=1}}f(m)^2|a_L(m)|^2\nonumber\\
&\qquad\qquad\qquad\qquad\geq \sum_{k\leq N}\frac{f(k)|a_L(k)|^2}{k^\sigma}\sum_{\substack{m\leq N/k\\ \gcd(k,m)=1}}f(m)^2|a_L(m)|^2\nonumber\\
&\qquad\qquad\qquad\qquad=\sum_{k\leq N}\frac{f(k)|a_L(k)|^2}{k^\sigma}\prod_{p\nmid k}(1+f(p)^2|a_L(p)|^2)\label{eq:mkMain}\\
&\qquad\qquad\qquad\qquad\quad+\mathcal{O}\left(\sum_{k\leq N}\frac{f(k)|a_L(k)|^2}{k^\sigma}\sum_{\substack{m> N/k\\ \gcd(k,m)=1}}f(m)^2|a_L(m)|^2\right).\nonumber
\end{align}
Applying the so-called Rankin's trick for $\alpha>0$ we get that the error term above is
\begin{align*}
&\ll\frac{1}{N^\alpha}\sum_{k\leq N}\frac{f(k)|a_L(k)|^2}{k^{\sigma-\alpha}}\prod_{p\nmid k}(1+p^\alpha f(p)^2|a_L(p)|^2)\\
&\ll\frac{1}{N^\alpha}\prod_{p}\left(1+\left( f(p)^2+\frac{f(p)}{p^{\sigma}}\right)p^\alpha|a_L(p)|^2\right).
\end{align*}
Moreover the main term in \eqref{eq:mkMain} is 
\begin{multline*}
\prod_{p}\left(1+\left( f(p)^2+\frac{f(p)}{p^{\sigma}}\right)|a_L(p)|^2\right)\\+\mathcal{O}\left(\frac{1}{N^\alpha}\prod_{p}\left(1+\left( f(p)^2+\frac{f(p)p^\alpha}{p^{\sigma}}\right)|a_L(p)|^2\right)\right).
\end{multline*}
Therefore
\begin{align}\label{eq:ratio}
\sum_{mk=n\leq N}\frac{a_L(k)r(m)\overline{r(n)}}{k^\sigma}&=\prod_{p}\left(1+\left( f(p)^2+\frac{f(p)}{p^{\sigma}}\right)|a_L(p)|^2\right)\\
&\quad+\mathcal{O}\left(\frac{1}{N^\alpha}\prod_{p}\left(1+\left( f(p)^2+\frac{f(p)}{p^{\sigma}}\right)p^\alpha|a_L(p)|^2\right)\right).\nonumber
\end{align}

Note that
\[
\sum_{n\leq N}|a_L(n)|^2 f(n)^2\leq \prod_{p}\left(1+f(p)^2|a_L(p)|^2\right),
\]
so our main purpose is to find suitable $f(n)$ such that the ratio between the error term and the main term in \eqref{eq:ratio} is $o(1)$ and
\begin{equation}\label{eq:largeRatio}
\prod_{p}\left(1+\left( f(p)^2+\frac{f(p)}{p^{\sigma}}\right)|a_L(p)|^2\right)\Bigg/ \prod_{p}\left(1+f(p)^2|a_L(p)|^2\right)
\end{equation}
is large.\\

First, let us consider the case $\sigma>1/2$. For any $L$, $M$ depending on $N$ and satisfying $L=o(M)$ we put
\[
f(p)=\begin{cases}
(L/p)^\sigma\ne 0&\text{ if $p\in[cL,M]$},\\
0&\text{ otherwise},
\end{cases}
\] 
where the choice of the positive constant $c$ will be optimized later. From \eqref{eq:SNCthm} we know that there exits $E(x)$ tending to $0$ as $x\to\infty$ such that
\[
\sum_{p\leq x}|a_L(p)|^2 = \kappa \frac{x}{\log x}(1+E(x)).
\]
Then, taking
\[
M = L\left(\min\left(\frac{1}{\max_{x>L}\sqrt{|E(x)|}},\frac{\log L
}{\log\log L}\right)\right)^{1/(2\sigma-1)}=:Lg(L)^{1/(2\sigma-1)},
\]
one can easily get
\begin{align*}
\sum_{p}f(p)^2|a_L(p)|^2\log p &= L^{2\sigma}\sum_{cL\leq p\leq X}\frac{|a_L(p)|^2\log p}{p^{2\sigma}}\\
& =\frac{\kappa c^{1-2\sigma}}{2\sigma-1} L - \frac{\kappa}{2\sigma-1}\frac{L}{g(L)}+o\left(\frac{L}{g(L)}\right).
\end{align*}
Therefore, for $\alpha = (\log L)^{-3}$, the ratio of the error term to the main term in \eqref{eq:ratio} is
\begin{align*}
&\leq N^{-\alpha}\prod_{p}\left(1+\left(f(p)^2+\frac{f(p)}{p^\sigma}\right)(p^{\alpha}-1)|a_L(p)|^2\right)\\
&\leq \exp\Bigg(-\alpha\log N + \alpha\sum_p f(p)^2|a_L(p)|^2\log p+\alpha\sum_p \frac{f(p)\log p}{p^\sigma}|a_L(p)|^2\\
&\quad+\mathcal{O}\left(\alpha^2\sum_p f(p)^2(\log p)^2|a_L(p)|^2\right)\Bigg)\\
&\leq \exp\left( -\alpha\frac{\kappa}{2\sigma-1}\frac{L}{g(L)}+o\left(\alpha\frac{L}{g(L)}\right)\right)=o(1),
\end{align*}
where the last inequality holds if 
\begin{equation}\label{eq:defL}
L = \frac{(2\sigma-1)c^{2\sigma-1}}{\kappa}\log N.
\end{equation}

Now, in order to estimate the ratio in \eqref{eq:largeRatio} it suffices to observe that $f(p)^2|a_L(p)|^2 \leq c^{-2\sigma}m^2$ and $f(p)|a_L(p)|^2/p^\sigma = o(1)$. Then the ratio in \eqref{eq:largeRatio} is
\begin{align*}
&\geq\exp\left(\frac{1+o(1)}{1+c^{-2\sigma}m^2}\sum_p\frac{f(p)|a_L(p)|^2}{p^\sigma}\right)\\
&=\exp\left(\left(\frac{\kappa c}{(2\sigma-1)(c^{2\sigma}+m^2)}+o(1)\right)\frac{L^{1-\sigma}}{\log L}\right)\\
&=\exp\left(\left(\kappa^\sigma(2\sigma-1)^{-\sigma}\frac{c^{2\sigma(3/2-\sigma)}}{c^{2\sigma}+m^2}+o(1)\right)\frac{(\log N)^{1-\sigma}}{\log \log N}\right)\\
&=\exp\left(\left(\kappa^\sigma m^{1-2\sigma}\frac{(3-2\sigma)^{3/2-\sigma}}{2 (2\sigma-1)^{1/2}}+o(1)\right)\frac{(\log N)^{1-\sigma}}{\log \log N}\right)
\end{align*}
since the optimal choice for $c$ is
\[
c = \left(m^2\frac{3-2\sigma}{2\sigma-1}\right)^{\frac{1}{2\sigma}}.
\]
This completes the proof in the case $\sigma>1/2$.\\

Now we assume that $\sigma=1/2$, and we define
\[
f(p) = 
\begin{cases}
\frac{\sqrt{L}}{\sqrt{p}\log p},&\text{if $p\in [L,M]$;}\\
0&\text{otherwise},
\end{cases}
\]
where, as before, $L$ and $M$ depend on $N$ and $L=o(M)$. As in the case $\sigma>1/2$ one can show that the ratio of the error term to the main term in \eqref{eq:ratio} is $o(1)$, provided that $M/L$ has sufficiently small order of growth depending on $E(x)$ and
\[
L = \kappa^{-1}\log N\log\log N.
\]
For this choice of $L$, by the fact that $f(p)^2|a_L(p)|^2=o(1)$, the ratio in \eqref{eq:largeRatio} is
\[
\geq \exp\left((1+o(1))\sqrt{L}\sum_p \frac{|a_L(p)|^2}{p\log p}\right) = \exp\left((1+o(1))\sqrt{\frac{\kappa{N}}{\log\log N}}\right),
\]
and the proof is complete.
\end{proof}

\section{Inhomogeneous Diophantine approximation. Proof of Theorem \ref{th:2}.} \label{sec_th2}

The classical Kronecker approximation theorem states that for real numbers
$\alpha_1, \dots, \alpha_n$ and for real numbers $\lambda_1, \dots, \lambda_n$
which are linearly independent over the rationals, for every given $\ve>0$ there
exists a real number $t$ such that 
$$
\|\lambda_k t - \alpha_k \| \leq \ve,
$$
where $\| \cdot \|$ denotes the distance to the nearest integer. A quantitative
form of this theorem in the homogeneous case $\alpha_1 = \dots = \alpha_n = 0$ is
at the core of Montgomery's proof of \eqref{mont}, and similarly a quantitative form in the inhomogeneous case, due
to Weber, is at the core of the argument of Pa{\'n}kowski and
Steuding. However, when carefully examining the argument in
\cite{PS} it turns out that what is required is not necessarily an
approximation result for the $\ell^\infty$ distance (which is represented by the
norm $\| \cdot \|$), but that an approximation result which holds ``on average''
in a certain sense is also suitable for this purpose. The subsequent lemma,
which is due to Chen \cite{chen}, provides such a result for the $\ell^2$-error in
inhomogeneous Diophantine approximation. The result is stated in a
multidimensional form in Chen's paper, but we only require it in the
one-dimensional setting. In the statement of the lemma, $M$ denotes a positive
integer.

\begin{lemma}[{\cite[Theorem 1 (i)]{chen}}] \label{lemmachen}
Let $\lambda_1, \dots, \lambda_n$ and $\beta_1, \dots, \beta_n$ be real
numbers, and assume that they have the property that for all integers $u_1,
\dots, u_n$ with $|u_j| \leq M$ the fact that
$$
u_1 \lambda_1 + \dots + u_n \lambda_n = 0
$$
implies that
$$
u_1 \beta_1 + \dots + u_n \beta_n \qquad \textrm{is an integer}.
$$
Then for all positive real numbers $\delta_1, \dots, \delta_n$ and for all real
numbers $T_1 < T_2$ we have
$$
\inf_{t \in [T_1,T_2)} \sum_{j=1}^n \delta_j \|\lambda_j t - \beta_j\|^2 \leq
\frac{\Delta}{4} \sin^2 \left(\frac{\pi}{2 (M+1)}\right) + \frac{\Delta M^n}{4 \pi \Lambda
(T_2-T_1)},
$$
where
$$
\Delta = \sum_{j=1}^n \delta_j
$$
and
\begin{eqnarray*}
\Lambda & = & \min \Big\{|u_1 \lambda_{1}+ \dots + u_n \lambda_{n}|:~u_j
\textrm{~are integers with~} \\
& & \qquad \qquad \qquad \qquad |u_j| \leq M,~1 \leq j \leq n, \quad
\textrm{and} \quad u_1 \lambda_1 + \dots + u_n \lambda_n \neq 0.\Big\}
\end{eqnarray*}
\end{lemma}

The subsequent lemma follows from a combination of \cite[Corollary 4.2]{PS}
with equation (7) of \cite{PS}. In the statement of the lemma and in the sequel, $\delta$
denotes the number from axiom (iv) of the definition of the Selberg class (where
it is assumed that $\delta<1/2$) and the numbers $\omega_p$ are such that $a_L(p) = |a_L(p)| e^{i \omega_p}$.

\begin{lemma} \label{lemma42}

Let $s_0 = \sigma+it_0$, and assume that $\sigma \in [1/2,1)$ and $t_0 \geq
15$. Furthermore, assume that $L(\sigma' + it) \neq 0$ for $\sigma' > \sigma$ and $|t-t_0|
\leq 2 (\log t_0)^2$. Then for $x > 2$ we have
\begin{align} 
\Re e^{-i\theta} \log L(s_0 + i t_1) & \geq \frac{1}{2} \sum_{\left| \log
\frac{p}{x} \right| \leq 1} \frac{|a_L(p)|}{p^{\sigma}} \cos (t_0 \log p -
\omega_p) \left(1 - \left| \log \frac{p}{x} \right|\right) \label{lemma42equ}\\
& \quad + \mathcal{O} \left(
2 x(\log t_0)^{-2} \right) + \mathcal{O} \left( x^{2\delta-2\sigma+1/2} \log x \right) \nonumber
\end{align}
for some $t_1 \in \left[-(\log t_0)^2,(\log t_0)^2\right]$.\footnote{The second error term, which contains the contribution of the numbers $n = p^l$ for $l \geq 2$, is given as $x^\delta/(\log x)$ in \cite{PS}. However, this is not necessarily smaller than the main term, which we will show to be of size roughly $x^{1-\sigma}$. Still, some calculations show that the error term can actually be bounded by $\mathcal{O} \left( x^{2 \delta - 2 \sigma + 1/2} \log x \right)$. Since by assumption $\delta<1/2$ we have $2 \delta - 2 \sigma + 1/2 < 1 - \sigma$, as necessary. In the following lines we show how to obtain this upper bound. We have to give an upper bound for $$
\sum_{\substack{|\log \frac{n}{x}| \leq 1,\\n = p^l,~l \geq 2}} \frac{|b(n)|}{n^{\sigma}}.
$$ 
Obviously we can assume that $l \ll \log x$. We know (axiom (iv)) that $|b(p^k)| \ll p^{k\delta}$. Let's start with $l=2$. The contribution is at most $$
\sum_{e^{-1} x \leq p^2 \leq ex} \frac{|b(p^2)|}{(p^2)^{\sigma}} \ll x^{1/2} \frac{x^{2 \delta}}{x^{2 \sigma}} \ll x^{2 \delta - 2 \sigma + 1/2}.
$$
In a similar way, for $l=3$ we get a contribution of at most
$$
x^{3 \delta - 3 \sigma + 1/3} \ll x^{2 \delta - 2 \sigma + 1/2}, 
$$
since $\delta < \sigma$. We get similar bounds for the contribution for larger values of $l$, and, as noted, we can assume that $l \ll \log x$. Thus the total error is at most
$$
x^{2 \delta - 2 \sigma + 1/2} \log x.
$$}
\end{lemma}
~\\

\begin{proof}[Proof of Theorem~\ref{th:2}]
Throughout the rest of this section we assume that all the conditions required for the validity of Theorem~\ref{th:2} are satisfied. Let $T$ be given, and assume that $T$ is ``large''. We will apply Lemma \ref{lemma42} with 
$$
x = B \log T,
$$
where $B$ is a positive number that will be chosen later. Let $p_1, \dots, p_n$
denote the primes in the interval $[x/e,ex]$. Then we can write the sum in
\eqref{lemma42equ} as
\begin{equation*}
\sum_{j=1}^n \frac{|a_L(p_j)|}{p_j^{\sigma}} \cos (t_0 \log p_j -
\omega_{p_j}) \left(1 - \left| \log \frac{p_j}{x} \right|\right) +  \mathcal{O} \left(
2 x(\log t_0)^{-2} \right) + \mathcal{O} \left( x^{\delta-\sigma+1/2} \log x \right)
\end{equation*}
We will use Lemma \ref{lemmachen} with $M=4$, 
$$
\lambda_j = \frac{\log p_j}{2\pi}, \qquad \beta_j = \frac{\omega_j}{2\pi}, 
$$
and
$$
\delta_j = \frac{|a_L(p_j)|}{p_j^{\sigma}}\left(1 - \left| \log \frac{p_j}{x}
\right|\right).
$$
Then the first condition of the lemma is satisfied due to the linear
independence of the logarithms of the primes. For the number $\Lambda$ we get
the lower bound
$$
e^{2 \pi \Lambda} \geq \frac{\left(\prod_{j=1}^n p_j\right)^{4} +
1}{\left(\prod_{j=1}^n p_j\right)^{4}},
$$
which implies that
\begin{equation} \label{Lambda}
\Lambda \gg e^{-(1+\ve) e x} = e^{-(1+\ve) 4 e B \log T}
\end{equation}
by the prime number theorem (for any fixed $\ve > 0$). Thus by Lemma \ref{lemma42} for any two numbers $T_1 < T_2$ we have
$$
\inf_{t \in [T_1,T_2)} \sum_{j=1}^n \delta_j \left\|\frac{t \log p_j - \omega_j}{2 \pi} \right\|^2 \leq
\frac{\Delta}{4} \sin^2 \left(\frac{\pi}{10}\right) + \frac{4^n \Delta}{4 \pi
\Lambda
(T_2-T_1)},
$$
where $\Delta = \sum_{j=1}^n \delta_j.$ Note that
$$
\cos y \geq 1 - 2 \pi^2 \left\| \frac{y}{2 \pi} \right\|^2, \qquad y \in
\mathbb{R}.
$$
Thus we obtain
\begin{eqnarray*}
\sum_{j=1}^n \delta_j \cos (t_0 \log p_j - \omega_{p_j}) & \geq & \Delta 
\underbrace{\left(1 - \frac{\pi^2}{2} \sin^2
\left(\frac{\pi}{10}\right)\right)}_{> 0.52} - \frac{4^n \Delta}{4 \pi \Lambda
(T_2-T_1)},
\end{eqnarray*}
and, if we can guarantee that
\begin{equation} \label{delta}
\frac{4^n}{4 \pi \Lambda
(T_2-T_1)} \leq \frac{1}{100},
\end{equation}
then we have
\begin{equation} \label{delta2}
\sum_{j=1}^n \delta_j \cos (t_0 \log p_j - \omega_{p_j}) \geq 0.51 \Delta.
\end{equation}

Choose $\mu < \eta$, where $\eta$ is the number from \eqref{zerod}, and
assume that $B$ satisfies $4 e B < \mu$. Furthermore, define
$$
T^{(r)} = \left[T + (r-1)T^\mu, T + rT^\mu \right), \qquad 1 \leq r \leq
T^{1-\mu}.
$$
Then by \eqref{Lambda} and since $n \ll (\log T)/(\log \log T)$ for sufficiently large $T$ we have
$$
T^{\mu} \geq 100\frac{4^n}{4 \pi \Lambda},
$$
which means that \eqref{delta} holds for $T_1$ and $T_2$ being the left and
right endpoints of an interval $T^{(r)}$, respectively. Furthermore, by
\eqref{zerod}, for sufficiently large $T$ there exists an index $\overline{r}
\in [2,T^{1-\mu}-1]$ such that $L(\sigma' + it) \neq 0$ for $\sigma' > \sigma$ and $t \in
\left( T^{\overline{r}-1} \cup T^{\overline{r}} \cup T^{\overline{r}+1}
\right)$. Thus by Lemma \ref{lemma42} and \eqref{delta2} we have
$$
\Re e^{-i\theta} \log L(s_0 + i t_1) \geq 0.51 \Delta + \mathcal{O} \left(
2 x(\log t_0)^{-2} \right) + \mathcal{O} \left( x^{\delta-\sigma+1/2} \log x \right)
$$
for some $t_0 \in T^{\overline{r}}$ and $t_1 \in \left[-(\log t_0)^2,(\log
t_0)^2\right]$. In particular we have
\begin{equation} \label{lowerb}
\max_{T \leq t \leq 2T} \Re e^{-i\theta} \log L(\sigma + i t) \geq 0.505 \Delta,
\end{equation}
provided that $T$ is sufficiently large.\\

It remains to give a lower bound for $\Delta$. We have
$$
\Delta = \sum_{\left| \log \frac{p}{B \log T} \right| \leq 1}
\frac{|a_L(p)|}{p^{\sigma}}\left(1 - \left| \log \frac{p}{B \log T}
\right|\right),
$$
In this sum everything is non-negative. Thus a lower bound for $\Delta$ is
\begin{eqnarray}
\sum_{|\log \frac{p}{B \log T}| \leq 1/2} \frac{|a_L(p)|}{p^{\sigma}} \underbrace{\left(1 - \left| \log \frac{p}{B \log T} \right| \right)}_{\geq 1/2} \geq \sum_{|\log \frac{p}{B \log T}| \leq 1/2}  \frac{|a_L(p)|}{2 (e^{1/2} B \log T)^{\sigma}}. \label{sumsigma0}
\end{eqnarray}
By \eqref{prime} we have
$$
\sum_{|\log \frac{p}{B \log T}| \leq 1/2} |a_L(p)| \sim \left( e^{1/2} - e^{-1/2} \right) \kappa \frac{B \log T}{\log \log T}.
$$
Combining this with \eqref{lowerb} and \eqref{sumsigma0} and choosing $\mu$ and $B$ such that $B$ is only slightly smaller than $\eta/(4 e)$  we obtain
\begin{eqnarray*}
& & \max_{T \leq t \leq 2T} \Re e^{-i\theta} \log L(\sigma + i t) \\
& \geq & 0.505 \left( e^{1/2} - e^{-1/2} \right) (1+o(1)) \kappa \frac{B \log T}{\log \log T} \frac{1}{2 (e^{1/2} B \log T)^{\sigma}} \\
& \geq & \frac{\left(1 - e^{-1} \right) \kappa}{4} \left(\frac{\eta}{4 \sqrt{e}} \right)^{1-\sigma} \frac{(\log T)^{1-\sigma}}{\log \log T},
\end{eqnarray*}
for sufficiently large $T$, which proves Theorem \ref{th:2}.
\end{proof}

\section{Upper bounds.} \label{sec_prop1}

This section deals with upper bounds for possible large values of $L$-functions. First, we will discuss what one can possibly gain by adopting our method from Section \ref{sec_th1} and using a different resonator function $r(n)$. It will turn out that by doing so we may only improve the constant $C_L(\sigma)$ in Theorem \ref{th:1}. In other words, we cannot expect to get a lower bound greater than $\max_{t\in[T,2T]}|L(\sigma+it)|\geq\exp\left(c(\log T)^{1-\sigma}/(\log\log T)^{\theta(\sigma)}\right)$, unless we carry out some significant modifications of the proof. We shall focus only on the case $\sigma>1/2$, for which it is more reasonable to ask for possible improvements. Nevertheless, one can easily adopt this argument for $\sigma=1/2$ to get similar conclusion.\\

Indeed, using the notation from Section \ref{sec_th1} it suffices to estimate the ratio
\begin{equation} \label{quotient1}
\left|\sum_{mk=n\leq N}\frac{a_L(k)r(m)\overline{r(n)}}{k^\sigma}\right|\Bigg/\sum_{n\leq N}|r(n)|^2.
\end{equation}

Note that for every positive real function $g(k)$ we have $2|r(m)r(mk)|\leq |r(mk)|^2/g(k)+g(k)|r(m)|^2$. Thus for any such function $g$ the numerator of \eqref{quotient1} is bounded above by
\begin{equation}\label{eq:I1upper}
\frac{1}{2}\sum_{n\leq N}|r(n)|^2\left(\sum_{k\leq N/n}\frac{g(k)|a_L(k)|}{k^\sigma}+\sum_{k|n}\frac{|a_L(k)|}{k^\sigma g(k)}\right).
\end{equation}

Now let us put
\[
g(k) = \begin{cases}|a_L(k)|f(k),&a_L(k)\ne 0,\\f(k)&\text{otherwise,}\end{cases}
\]
where $f(k)$ is a multiplicative function such that $f(p^k)=\min(1,(L/p^k)^\beta)$ with $1-\sigma<\beta<\sigma$ and $L=\log N$.

Then, noticing that $f(p^k)\leq f(p)$ for every prime $p$ and any positive integer $k$, the assumption of the Selberg normality conjecture \eqref{eq:SNCthm} gives
\begin{align}
\sum_{k\leq N/n}\frac{g(k)|a_L(k)|}{k^\sigma}&\leq\prod_p\left(1+\sum_{k\geq 1}\frac{|a_L(p^k)|^2f(p^k)}{p^{k\sigma}}\right) \nonumber\\
&\leq\exp\left(\sum_{p\leq L}\frac{|a_L(p)|^2}{p^\sigma}+L^\beta\sum_{p> L}\frac{|a_L(p)|^2}{p^{\sigma+\beta}}+\mathcal{O}(1)\right) \nonumber\\
&\ll\exp\left((\kappa+o(1))\left(\frac{L^{1-\sigma}}{(1-\sigma)\log L}+\frac{L^{1-\sigma}}{(\sigma+\beta-1)\log L}\right)\right) \nonumber\\
&\ll\exp\left((\kappa+o(1))\frac{\beta(\log N)^{1-\sigma}}{(1-\sigma)(\sigma+\beta-1)\log\log N}\right). \label{eq:I2upper}
\end{align}

Next, from the definition of $f(k)$, we have for $n\leq N$ that
\begin{align*}
\sum_{k|n}\frac{|a_L(k)|}{k^\sigma g(k)}&\leq \prod_{p^a||n}\left(1+\sum_{1\leq j\leq a}\frac{1}{L^\beta p^{(\sigma-\beta)j}}\right)\prod_{p|n}\left(1+\frac{1}{p^\sigma-1}\right)\\
&\leq \exp\left(L^{-\beta}\sum_{p|n}\frac{1}{p^{(\sigma-\beta)}-1}+\sum_{p|n}\frac{1}{p^\sigma-1}\right)\\
&=\exp\left((1+o(1))\frac{(2-2\sigma+\beta) (\log N)^{1-\sigma}}{(1-\sigma)(1-\sigma+\beta)\log\log N}\right).
\end{align*}
Therefore, from \eqref{eq:I1upper}, \eqref{eq:I2upper}, and the definitions of $M_1$ and $M_2$ we see that this method can ensure the existence of large values of $L$-functions of size at most
\[
\exp\left(c\frac{(\log T)^{1-\sigma}}{\log\log T}\right).
\]
Thus our resonator function was already chosen in a way which is essentially optimal; this is also in accordance with the results in \cite{H}. The only possibility for a significant improvement seems to be to increase the value of $N$ far beyond $T^{1-\varepsilon}$, as in \cite{A,BS} in the context of large values of the Riemann zeta function. However, the method of constructing a ``sparse'' extremely long resonator cannot be transferred from the Riemann zeta function to general L-functions, since it depends crucially on the fact that all coefficients in the Dirichlet series representation of $\zeta$ are positive reals.\\

Now let us prove Proposition \ref{prop:1}, which states that, in general, it is impossible to find large values greater than $\exp\left(c(\log T)^{2-2\sigma}/(\log\log T)\right)$ as long as we assume the truth of an analogue of the Riemann hypothesis for a given $L$-function. 

\begin{proof}[Proof of Proposition \ref{prop:1}]
We closely follow the proof of \cite[Theorem 14.5]{T}, where it is shown that our assertion holds for the Riemann zeta-function. Hence we shall be very sketchy (see also \cite[p.74--75]{CG1}).\\

It is easy to see that using the Borel--Carath\'eodory theorem one has
\[
\max_{|z-2-it|<3/2-\delta'}|\log L(z)|\ll \frac{1}{\delta'}\log t,
\]
since $\Re\log L(z)\ll \log t$ if $|z-2-it|<(3-\delta')/2$. Moreover, since $a_L(p)=b(p)$ and $b(p^k)\ll p^{k\delta}$ for some $\delta<1/2$, assumption \eqref{prime_prop} implies that
\begin{align*}
\max_{x>1+\eta}|\log L(x+it)| &\leq \sum_p\sum_{k\geq 1}\frac{|b(p^k)|}{p^{k(1+\eta)}}\leq \sum_p \frac{|a_L(p)|}{p^{1+\eta}}+\sum_p\sum_{k\geq 2}\frac{1}{p^{k(1-\delta)}}\\
&\leq (1+\eta)\int_2^\infty \frac{\sum_{p\leq u}|a_L(p)|}{u^{2+\eta}} + \mathcal{O}(1)  \ll \frac{1}{\eta},
\end{align*}
provided $\eta$ is sufficiently small.  Hence, using Hadamard's three-circles theorem and taking $\delta'=\eta=(\log\log t)^{-1}$ we obtain that
\begin{equation}\label{eq:fromHadamard}
\log L(\sigma+it)\ll (\log t)^{2-2\sigma}\log\log t\qquad\text{for $\ \frac{1}{2}+\frac{1}{\log\log t}\leq \sigma \leq 1$.}
\end{equation}

Now, recall that Kaczorowski and Perelli \cite{KP} proved that
\[
N_L(T) = \frac{d_L}{2}T\log T + c_L T+\mathcal{O}(\log T)\quad\text{for some $c_L>0$},
\]
where $N_L(T)$ counts the non-trivial zeros $\rho=\beta+i\gamma$ of $L(s)$ with $|\gamma|\leq T$. Moreover, as it was shown in \cite[Lemma 4]{AM} or \cite[Lemma 5]{Mu}, for $-5/2\leq \sigma\leq 7/2$,
\[
\frac{L'(s)}{L(s)} = \sum_{|t-\gamma|\leq 1}\frac{1}{s-1/2-i\gamma}+\mathcal{O}(\log|t|).
\]
Since the number of terms in the sum is $\ll \log t$, we obtain
\[
\frac{L'(s)}{L(s)}\ll \log t\qquad\text{if $\sigma\ne \frac{1}{2}$},
\]
and
\[
\frac{L'(s)}{L(s)}\ll \frac{\log t}{\min\{|t-\gamma\}}+\log t\qquad\text{uniformly for $-\frac{5}{2}\leq \sigma\leq \frac{7}{2}$}.
\]
Hence, for each interval $(n,n+1)$ we can find $t_n$ such that
\[
\frac{L'(s)}{L(s)}\ll (\log t)^2\qquad\text{uniformly for $-\frac{5}{2}\leq \sigma\leq \frac{7}{2}$,\ $t=t_n$}.
\]

Now let $\Lambda_L(n)$ denote the coefficients of $-{L'(s)}/{L(s)}$. Then, using a method of contour integration one can show that
\begin{align}
\sum_n\frac{\Lambda_L(n)}{n^s}e^{-\delta n} &=-\frac{1}{2\pi i}\int_{(2)}\Gamma(z-s)\frac{L'(z)}{L(z)}\delta^{s-z}dz\label{eq:contour}\\
&= -\frac{1}{2\pi i}\int_{(1/4)}\Gamma(z-s)\frac{L'(z)}{L(z)}\delta^{s-z}dz\nonumber\\
&\quad-\frac{L'(s)}{L(s)} -\sum_\rho \Gamma(\rho-s)\delta^{s-\rho}+\mathcal{O}(e^{-ct}).\nonumber
\end{align}
Hence, short calculations give 
\begin{equation} \label{twosums}
-\frac{L'(s)}{L(s)} = \sum_n\frac{\Lambda_L(n)}{n^s}e^{-\delta n} + \sum_\rho \Gamma(\rho-s)\delta^{s-\rho} + \mathcal{O}(\delta^{\sigma-1/4}\log t).
\end{equation}
Applying again the fact that $N_L(T+1)-N_L(T)\ll \log T$ gives that the second sum on the right-hand side of \eqref{twosums} is $\ll \delta^{\sigma-1/2}\log t$. In order to estimate the first sum on the right-hand side of \eqref{twosums} we use again \eqref{eq:contour} and move the path of integration to $\Re(z) = \sigma$. Then we have
\[
\sum_n\frac{\Lambda_L(n)}{n^s}e^{-\delta n}\ll \delta^{\sigma-1}.
\]
Taking $\delta=(\log t)^{-2}$ gives
\[
\frac{L'(s)}{L(s)} \ll (\log t)^{2-2\sigma},
\]
which, together with \eqref{eq:fromHadamard}, easily implies that
\[
\log L(s) \ll \frac{(\log t)^{2-2\sigma}}{\log\log t}.
\]

In order to get the upper bound for $\sigma=1$ it suffices to integrate \eqref{twosums} over the interval $[1,7/2]$ with $\delta=(\log t)^{-2}$. Then, by Ramanujan's conjecture and the fact that $a_L(p)\ll p^\varepsilon$, we have
\begin{align*}
\log L(1+it) &\leq \sum_{n\leq N} \frac{|a_L(p)|}{p} + O\left(\sum_{n> N}e^{-\delta n}\right) + \mathcal{O}(1)\\
&\leq \sum_{n\leq N} \frac{|a_L(p)|}{p} + O\left(\frac{1}{\delta e^{\delta N}}\right) + \mathcal{O}(1).
\end{align*}
From partial summation and \eqref{SPNT}, one can easily see that the first sum is $\leq \kappa\log\log N + \mathcal{O}(1)$. Thus taking $N = 1+[\log^3 t]$  completes the proof.
\end{proof}

\section{Further application of Chen's theorem. Proof of Theorem \ref{th:3}.} \label{sec_th3th4}

First note that for $\sigma>1$ we have
\[
\Re e^{-i\theta} \log L(s) \geq \sum_{p\leq x}\frac{|a_L(p)|}{p^{\sigma}}\cos\left(t\log p +\theta-\omega_p\right) - \sum_{p>x}\frac{|a_L(p)|}{p^\sigma} + \mathcal{O}(1),
\]
since $b(p^j)\ll p^{j\delta}$ for some $\delta<1/2$. 

Now, let $p_1,\ldots,p_n$ denote all primes not exceeding $x$. Then, as in the proof of Theorem \ref{th:2}, we use Lemma \ref{lemmachen} with $M^2=\log\log\log T$, $T_1=T$, $T_2=2T$, 
$$
\lambda_j = \frac{\log p_j}{2\pi}, \qquad \beta_j = \frac{\omega_{p_j}-\theta}{2\pi} \qquad\text{and}\qquad 
\delta_j = \frac{|a_L(p_j)|}{p_j^{\sigma}}.
$$
Then, for $B = 1/(2M)$, we get $M^n\Lambda^{-1}/T = \mathcal{O}(T^{-1/4})$ and 
\[
\max_{T\leq t\leq 2T} \sum_{p\leq x}\frac{|a_L(p)|}{p^{\sigma}}\cos\left(t\log p + \theta-\omega_p\right)\geq (1+\mathcal{O}(M^{-2}))\sum_{p\leq x}\frac{|a_L(p)|}{p^\sigma}.
\]

Next, one can easily get from \eqref{SPNT}  and the classical second mean value theorem that
\begin{align*}
\sum_{p>x}\frac{|a_L(p)|}{p^\sigma} &\leq \lim_{y\to\infty}(1+\sigma)\int_x^y\frac{\sum_{p\leq u}|a_L(p)|}{u^{1+\sigma}}du + \mathcal{O}(1)\\
&\ll \frac{1}{\log x}\int_x^\infty u^{-\sigma} du \ll\frac{x^{1-\sigma}}{(\sigma-1)\log x},
\end{align*}
which is $\mathcal{O}(1)$ if $\sigma\geq 1+\frac{1}{2\log x}$ and $\sigma\ll 1$.

Moreover, by partial summation, we get
\begin{align*}
\sum_{p\leq x}\frac{|a_L(p)|}{p^\sigma} \geq \int_2^x \frac{\sum_{p\leq u}|a_L(p)|}{u^{\sigma+1}}du + \mathcal{O}(1),
\end{align*}
and again \eqref{SPNT} yields
\begin{align*}
\sum_{p\leq x}\frac{|a_L(p)|}{p^\sigma} &= \int_2^x \frac{\kappa}{u^{\sigma+1}\log u}du + O\left(\int_2^x\frac{du}{u\log^2 u}\right)+\mathcal{O}(1)\\
&= \kappa\li(x^{1-\sigma}) - \kappa\li(2^{1-\sigma}) + \mathcal{O}(1)= - \kappa\li(2^{1-\sigma}) + \mathcal{O}(1),
\end{align*}
where the last equality is fulfilled for $\sigma\geq 1+\frac{1}{2\log x}$.

Now, recall that it is well known (see for example \cite[Eq. (9)]{N}) that
\[
\li(\xi) = C+\log(-\log \xi) + \sum_{j=1}^\infty\frac{\log^j \xi}{j!j}
\]
for some positive constant $C$ and any $\xi$ with $0<\xi<1$. Therefore, $-\li(2^{1-\sigma}) = \log \frac{\log 2}{\sigma-1} + \mathcal{O}(1)$, and taking $\sigma=1+\frac{\log 2}{\log x}$ leads to
\[
\Re e^{i\theta} \log L(s) \geq (\kappa+\mathcal{O}(M^{-2})) \log\log x + \mathcal{O}(1),
\]
which proves \eqref{eq:Steuding} by recalling that $x= \frac{\log T}{2\sqrt{\log\log\log T}}$.

In order to show the second assertion, let us write $f(s) = L(s)/(\log\log s)^\kappa$ and assume that $f(1+it) = o(1)$. Obviously $f(2+it) = o(1)$. Therefore, by the Phragm\'en--Lindel\"of principle, we get a contradiction with \eqref{eq:Steuding}.
\section*{Acknowledgements.} 

We want to thank Titus Hilberdink and Maksym Radziwi{\l}{\l} for several helpful comments and discussions.

\end{document}